\documentclass{amsart}
\usepackage{amsmath}
\usepackage{framed,comment,enumerate}
\usepackage[pdftex]{graphicx}
\usepackage{epstopdf}
\usepackage{xcolor, framed}
\usepackage{color}
\usepackage{hyperref}
\usepackage{mathtools}
\mathtoolsset{showonlyrefs,showmanualtags}
\def\vem{\vspace{1em}}
\def\hem{\hspace{0.5em}}

\def\R {\mathbb{R}}
\def\N{\mathbb{N}}
\def\eps{\varepsilon}

\def\HPP{\{y=0\}}
\def\LSet{\Lambda(u)}
\def\ddt{\frac{\partial}{\partial y}}

\def\GC{\mathcal{G}^2_c}
\def\PC{\mathcal{P}^2_c}
\def\SC{\mathcal{S}_c}
\def\GCk{\mathcal{G}^k_c}
\def\PCk{\mathcal{P}^k_c}

\def\OLSet{\{U=0\}}
\def\OLSetn{\{U_n=0\}}
\def\UMY{U-\tfrac{y^2}{2}}
\def\UnMY{U_n-\tfrac{y^2}{2}}
\def\hW{\widehat{W}}

\renewcommand{\ell}{b}

\newcommand{\cH}{\ensuremath{\mathcal H}}

\DeclareMathAlphabet{\mathup}{OT1}{\familydefault}{m}{n}
\newcommand{\dx}[1]{\mathop{}\!\mathup{d} #1}
\newcommand{\dH}[1]{\mathop{}\!\mathup{d}\mathcal{H}^{#1}}
\newcommand{\cW}{\ensuremath{\mathcal W}}

\newtheorem{thm}{Theorem}[section]
\newtheorem{prop}[thm]{Proposition}

\newtheorem{cor}[thm]{Corollary}
\newtheorem{lem}[thm]{Lemma}
\theoremstyle{definition}
\newtheorem{defi}[thm]{Definition}
\newtheorem{rem}[thm]{Remark}
\numberwithin{equation}{section}

\title[Compact Contact Sets]{Compact contact sets of sub-quadratic solutions to the thin obstacle problem} 

\author{Simon Eberle}
\address{Basque Center for Applied Mathematics, Bilbao, Spain}
\email{seberle@bcamath.org}
\thanks{S.E. is supported by the European Research Council under Grant Agreement No 948029.}
\author{Hui Yu}
\address{Department of Mathematics,	National University of Singapore, Singapore}
\email{ huiyu@nus.edu.sg}

\begin{document}

\begin{abstract}
We study global solutions to the thin obstacle problem with at most quadratic growth at infinity. 

We show that every ellipsoid can be realized as the contact set of such a solution. On the other hand, if such a solution has a compact contact set, we show that it must be an ellipsoid. 

%This resolves a conjecture in \cite{ERW} in the negative. 
\end{abstract}

\maketitle
%%%%%%%%%%%%%%%%%%%%%%%%%%%%%%%%%%%%%%%%%%%%%%%%%%

\section{Introduction}
In this article, we consider \textit{global solutions to the thin obstacle problem}, that is, solutions to 
\begin{equation}
\label{TOP}
\begin{cases}
\Delta u\le 0 &\text{ in $\R^{d+1}$,}\\
u\ge 0 &\text{ on $\{y=0\}$,}\\
\Delta u=0 &\text{ in $\{y\neq 0\}\cup\{u>0\}$,}\\
u(\cdot,y)=u(\cdot,-y) &\text{ for all $y\in\R$}
\end{cases}
\end{equation}
 in the entire space
 \begin{equation}
 \label{EqnDecomposition}
 \R^{d+1}=\{(x,y):x\in\R^d,y\in\R\}
 \end{equation}
with $d\ge2.$ We impose the symmetry assumption $u(\cdot,y)=u(\cdot,-y)$ for simplicity. This poses no restriction on the class of solutions that we consider, as any solution can be symmetrized by taking its even part in the $y$-variable. 

This system of equations describe the height of an elastic membrane resting on a lower-dimensional obstacle. In this context, the region where the membrane is in contact with the obstacle is called the \textit{contact set}, namely, 
\begin{equation}
\label{CSet}
\Lambda(u):=\{u=0\}\cap\HPP.
\end{equation}

In the past few decades, the thin obstacle problem has been the subject of intensive research. See, for instance,  classical results in \cite{AC, ACS, GP,U} and recent developments in \cite{CSV, DS, FeR, FSe, FoS, KPS, SY1, SY2}. For a gentle introduction to the thin obstacle problem and its neighboring fields, the reader could consult the monograph by Petrosyan-Shahgholian-Uraltseva \cite{PSU}. 

Most of these results focus on local properties of  the solution, especially near a contact point. Instead, we aim to study solutions in the entire $\R^{d+1}$, and explore the rigidity in the global problem in terms of possible shapes of the contact set $\Lambda(u)$. 

\vem 

For the classical obstacle problem, such a program has recently been completed. Consider a global solution to the \textit{classical obstacle problem}
\begin{equation}
\label{OP}
\begin{cases}
\Delta U=\chi_{\{U>0\}} &\text{ in $\R^d$,}\\
U\ge0 &\text{ in $\R^d$,}
\end{cases}
\end{equation}
it was conjectured by Shahgholian \cite{Sh} and Karp-Margulis \cite{KM} that the contact set $\{U=0\}$, after removing invariant directions of the solution,  can only be a point, an ellipsoid or a paraboloid. 

This conjecture was first considered for the case when the contact set $\{U=0\}$ is compact. In this case, the conjecture was confirmed by the works of Dive \cite{Di}, DiBenedetto-Friedman \cite{DF} and Friedman-Sakai \cite{FSa}. A short proof was found by Eberle-Weiss \cite{EW}. 

Without the compactness assumption, the conjecture was solved in two dimensions by Sakai \cite{Sa} with complex variable techniques. In dimensions no lower than 6, it was confirmed by Eberle-Shahgholian-Weiss \cite{ESW}. Finally in Eberle-Figalli-Weiss \cite{EFW}, the conjecture was proved in its full generality, completing a program that has lasted for more than ninety years. 

\vem

Apart from Sakai's result in two dimensions, all of these works rely on the explicit relation between a domain and its Newtonian potential. With this relation, it is straightforward to construct  a solution whose contact set is a given ellipsoid. With a non-trivial fixed point argument by DiBenedetto-Friedman \cite{DF}, this relation also allows the construction of a solution with an ellipsoidal contact set and a prescribed blow-down profile. 

Such a relation no longer holds for the thin obstacle problem. As a result, not much is known about  possible contact sets of solutions to \eqref{TOP}, even in the compact regime. It is not clear how to construct a solution with a given ellipsoid as its contact set. Nor is it clear whether ellipsoids are the only possible compact contact sets of global solutions\footnote{It was conjectured by one of the authors with collaborators that ellipsoids are not the only possible compact contact sets. See Remark \ref{RemConjecture}.}. 

\vem

In this article, we answer these questions for a particular class of solutions to the thin obstacle problem \eqref{TOP}.

Without tools from potential theory, we  construct global solutions to \eqref{TOP} by approximating {with a linearization of} solutions to the classical obstacle problem \eqref{OP} around the polynomial $\tfrac{1}{2} y^2$. This linearization was already used in Figalli-Serra \cite{FiSe} and Savin-Yu \cite{SY3}.

This method poses a restriction on the type of solutions we consider. Since solutions to the classical obstacle problem have quadratic growth, we can only access solutions  that grow at most quadratically, that is, solutions satisfying:
\begin{equation}
\label{IntroSubQuad}
\sup_{(x,y)\in\R^{d+1}}\frac{|u(x,y)|}{|(x,y)|^2+1}<+\infty.
\end{equation} 

We first construct solutions to \eqref{TOP} with given ellipsoids as contact sets:
\begin{thm}
\label{IntroMainThm1}
For $d\ge2$,
given $\ell_j>0$ for $j=1,2,\dots,d$ and the corresponding thin ellipsoid
$$
E':= \Big \{ (x,y) \in \R^{d} \times \R: \sum_{1\le j\le d}  \tfrac{x_j^2}{\ell_j^2}  \le1, y = 0 \Big \}\subset\HPP,
$$
there is a global solution to the thin obstacle problem \eqref{TOP} with at most quadratic growth \eqref{IntroSubQuad} such that 
$$
\LSet=E'.
$$

Moreover, if $v$ is another solution with at most quadratic growth and $\Lambda(v)=E'$, then there is $c\in\R$ such that 
$$
v=cu.
$$
\end{thm} 
The contact set $\LSet$ is defined in \eqref{CSet}.

On the other hand, we show that for solutions with  sub-quadratic growth, ellipsoids are the only possible contact sets in the compact regime:
\begin{thm}
\label{IntroMainThm2}
For $d\ge2$,
suppose that $u$ is a global solution to the thin obstacle problem \eqref{TOP} with at most quadratic growth \eqref{IntroSubQuad}. 

If the contact set $\LSet$ is compact and has non-empty interior in the thin space $\HPP$, then $\LSet$ is a thin ellipsoid. 
\end{thm} 

\begin{rem}
\label{RemConjecture}
If the contact set $\LSet$ is compact and has empty interior in $\HPP$, then it is a point or the empty set. With this, we have a complete classification of compact contact sets for solutions satisfying \eqref{TOP} and  \eqref{IntroSubQuad}.

In \cite[Remark 2]{ERW}, it is conjectured that other shapes could appear as compact contact sets for such solutions. We resolve this conjecture in the negative. 
\end{rem} 

\begin{rem}
The corresponding result in the non-compact regime will be addressed in a future work. On the other hand,  the restriction on the growth rate \eqref{IntroSubQuad} seems essential with our method. 
\end{rem} 

This manuscript is structured as follows: In the next section, we clarify the notations. We  collect then some preliminaries in the third section.  In the fourth section, we construct solutions with prescribed ellipsoidal contact sets as in  Theorem \ref{IntroMainThm1}. In the fifth section, we   construct solutions with given polynomial expansions, which is used to prove Theorem \ref{IntroMainThm2}. The uniqueness of solution corresponding to each ellipsoidal contact set is proved in the sixth section.

\section*{Acknowledgements}
Part of this work was done while the authors were in residence at Institut Mittag-Leffler for the program on ``Geometric Aspects of Nonlinear Partial Differential Equations". They thank the organizers for organizing this wonderful program, and  the institute for its hospitality. 

%%%%%%%%%%%%%%%%%%%%%%%%%%%%%
%%%%%%%%%%%%%%%%%%%%%%%%%%%%%

\section{Notation}
We work in $\R^{d+1}$ with $d\ge2$,  equipped with the Euclidean inner product $X \cdot Y$ and the induced norm $|X|$. We often write $X \in \R^{d+1}$ as $$X= (x,y) \in \R^{d} \times \R.$$ 

The measure $\cH^{d}$ denotes the $d$-dimensional Hausdorff measure. 

For a function $f$ we denote by $f_+ := \max\{f,0\}$ its positive part and by $f_- := \max\{-f,0\}$ its negative part.

For a bounded, measurable set $E \subset \R^{d+1}$ we denote by $V_E$ the \textit{Newtonian potential} of $E$, which is defined  as
\begin{equation}\label{DefNT}
		V_E(X) := \kappa_{d+1} \int_E |X-Y|^{2-(d+1)} \dx{Y}, \qquad \kappa_{d+1} := \frac{1}{(d+1)(d-1) |B_1|}.
\end{equation}
This potential satisfies
\begin{equation}
\Delta V_E=-\chi_E \text{ in }\R^{d+1}.
\end{equation}
Here, as in the remaining part of the paper, the characteristic function of a set $E$ is denoted by $\chi_E$.

%%%%%%%%%%%%%%%%%%%%%%%%%%%%%%%%%%%%%%%%%%%%%%%%%%%%%%%%%%%%%%%%%%%%%%%%%%%%%%%%%%%%%%%%%%%%%%%%%%%%%%%%%%%%%%%%%%%%%%%%
\section{Preliminaries}
In this section, we collect some preliminaries, firstly about the thin obstacle problem \eqref{TOP}, then about the classical obstacle problem \eqref{OP}.

\subsection{Solutions to the thin obstacle problem}
Let $u$ be a solution to the thin obstacle problem \eqref{TOP}. By the work of Athanasopoulos-Caffarelli \cite{AC}, it is locally Lipschitz in the entire space, and locally $C^{1,\frac12}$ in the upper half-space $\{y\ge0\}$. 

Although the solution may fail to be differentiable at points on the thin space $\HPP$, it is convenient to use one-sided derivatives. With our decomposition of $\R^{d+1}$ as in \eqref{EqnDecomposition}, we define 
\begin{equation}
\label{EqnOneSidedD}
\ddt u(x,0):=\lim_{t\to 0^+}\frac{u(x,t)-u(x,0)}{t}
\end{equation} 
for a point $(x,0)$ in the hyperplane $\HPP.$

With this,  the non-positive measure $\Delta u$ can be identified as the following {\cite[§9.1.3]{PSU}}. Recall the definition of the contact set $\LSet$ from \eqref{CSet}.
\begin{lem}
\label{LemIdentifyDelta}
Suppose that $u$ is a solution to the thin obstacle problem \eqref{TOP}. Then 
$$
\Delta u=2 \frac{\partial u}{\partial y} \hem \cH^{d}|_{\LSet}.
$$
\end{lem}  

In particular, we have $$u\Delta u=0.$$ This allowed Athanasopoulos-Caffarelli-Salsa \cite{ACS} to apply Almgren's monotonicity formula to  the thin obstacle problem. 

For a point ${X}   \in\HPP$ and a scale $r>0$, \textit{Almgren's frequency function} is defined as 
\begin{equation}
\label{EqnAlmgren}
\Phi(u;{X},r):=\frac{r\int_{B_r({X})}|\nabla u|^2}{\int_{\partial B_r({X})}u^2 {\dH{d}}}.
\end{equation}
When centered at $0$, we omit the dependence on the point ${X}$ in the expression. 
This function is monotone:
\begin{prop}[{\cite[Lemma 1]{ACS}}]
Suppose that $u$ is a solution to the thin obstacle problem \eqref{TOP}, then
$$
\frac{\dx{}}{\dx{r}}\Phi(u;r)\ge 0.
$$
In particular, the following limit exists
$$
\lambda_\infty:=\lim_{r\to+\infty}\Phi(u;r).
$$
\end{prop} 

This limit $\lambda_\infty$ can be viewed as the \textit{frequency at infinity}. It controls the growth of the solution:
\begin{cor}
\label{CorGrowthRateFromAlmgren}
Let $u$ be a global solution to \eqref{TOP}. Then for $r\ge 1$, we have
$$
\int_{\partial B_r}u^2 {\dH{d}}\le  r^{d+2\lambda_\infty}  \int_{\partial B_1}u^2 {\dH{d}}.
$$
\end{cor} 
\begin{proof}
Define 
$$
H(r):=\frac{1}{r^{d+2\lambda_\infty}}\int_{\partial B_r}u^2 {\dH{d}}.
$$
Differentiating $H$ leads to 
\begin{equation}
\label{EqnDOfHeight}
\frac{\frac{\dx{}}{\dx{r}}H(r)}{H(r)}=\frac{2}{r}[\Phi(u;r)-\lambda_\infty]+\frac{2\int_{B_r}u\Delta u}{\int_{\partial B_r}u^2 {\dH{d}}}.
\end{equation}
With $u\Delta u=0$ for solutions to \eqref{TOP}, we have
$$
\frac{\dx{}}{\dx{r}}\log H(r)=\frac{2}{r}[\Phi(u;r)-\lambda_\infty].
$$
With the monotonicity of the frequency function, we have 
$$
\Phi(u;r)-\lambda_\infty\le 0 \text{ for all $r$.}
$$
This implies $H(r)\le H(1)$ for $r\ge 1$, the desired the estimate. 
\end{proof} 

In this work, we are interested in solutions with at most quadratic growth and with compact contact sets, that is, solutions in the following class:
\begin{defi}
\label{DefGC}
Suppose that $u$ is a solution to the thin obstacle problem \eqref{TOP} in $\R^{d+1}$ for $d\ge2$ with $\LSet\neq\emptyset.$ 

We say that $u$ has \textit{at most quadratic growth} if it satisfies
$$
\sup_{X\in\R^{d+1}}\frac{|u(X)|}{|X|^2+1}<+\infty.
$$ 
If, further, the contact set $\LSet$  in \eqref{CSet} is compact, we write
$$
u\in\mathcal{G}^2_c.
$$
\end{defi} 

%As observed in \cite{ERW}, solutions in $\GC$ are convex along directions tangential to $\HPP$. In particular, we have
%\begin{prop}
%For $u\in\GC$, the contact set $\LSet$ is convex. 
%\end{prop} 

In \cite{ERW}, it was shown that solutions to \eqref{TOP} with polynomial growth and compact contact sets form a finite dimensional space. This was done by establishing a bijection between solutions in this class and a natural class of polynomials. We only need a special case of their result, when the rate of growth is quadratic. 

In this case, the relevant {polynomials are quadratic and satisfy}
$$
p(\cdot, y)=p(\cdot, -y), \hem \Delta p=0,
\text{ and }
\{p(\cdot,0)\le 0\} \text{ is compact}.
$$
If a {quadratic polynomial} $p$ satisfies all these properties, we write
\begin{equation}
\label{EqnPC}
p\in\mathcal{P}_c^2.
\end{equation} 

The bijection between $\GC$ and $\PC$ is given by the following theorem:

\begin{thm}[{\cite[Theorem 1]{ERW}}]
\label{ThmBijection}
For each $u\in\GC$, there is $p\in\PC$ such that 
$$
u=p+v,
$$
where $v$ is the unique solution to the thin obstacle problem that vanishes at inifinity and has $-p$ as the obstacle, namely, 
$$
\begin{cases}
\Delta v\le 0 &\text{ in $\R^{d+1}$,}\\
v\ge-p &\text{ on $\HPP$,}\\
\Delta v=0 &\text{ in $\{y\neq0\}\cup\{v>-p\},$}\\
v(\cdot,y)=v(\cdot,-y)&\text{ for all $y\in\R$,}\\
v(x,y)\to0 &\text{ as $|(x,y)|\to+\infty.$}
\end{cases}
$$
This map $p\mapsto u$ is a bijection between $\PC$ and $\GC$. 
\end{thm} 
\begin{rem}
This theorem is established for operators of the form $\mathrm{div}(|y|^{a}\nabla\cdot)$ in \cite{ERW}. Such operators become very singular when the parameter $a$ is close to $-1$. To deal with this singularity, the authors have to restrict the dimension $d$ to be $d\ge3$.

For our purpose, we generalize their result to $d\ge2$. To avoid diluting the focus, we leave the proof to the appendix.
 \end{rem} 

\subsection{{Reduction to `normalized solutions'}}
\label{RemNormalizationOfBijection}
In the expansion $u=p+v$, 
 up to a rotation and a translation in the thin space $\HPP$, and  a multiplication by a positive number, we might assume that $p$ is of the form
$$
p=\frac12\sum_{1\le j\le d}a_jx_j^2-\frac12 y^2-c
$$
for $a_j>0$, $\sum_{1\le j\le d}a_j=1$ and ${c \in \R}$.
{If $c < 0$, then $\LSet$ is the empty set.}
If $c=0$, then $\LSet$ is a single point. 

If $c>0$, we consider quadratic rescalings $u_r(x,y):=\frac{1}{r^2}u(rx,ry)$, which have expansions of the form
$$
u_r=\frac12\sum_{1\le j\le d}a_jx_j^2-\frac12 y^2-\frac{c}{r^2}+\frac{1}{r^2}v(rx,ry).
$$
As a result, by picking $r=\sqrt{c}$, we only need to consider solutions with the expansion
$$
u=(\frac12\sum_{1\le j\le d}a_jx_j^2-\frac12 y^2-1)+v,
$$
where
$a_j>0, \text{ and } \sum_{1\le j\le d}a_j=1.$

%%%%%%%%%%%%%%%%%%%%%%%%%%%%%%%%%%%%%%%%%%%%%%%%%%%%%%%%%%%%
\subsection{Solutions to the classical obstacle problem}
 To simplify our exposition,  if $U$ is a global solution to the classical obstacle problem with a compact contact set, namely, 
$$
\begin{cases}
\Delta U=\chi_{\{U>0\}} &\text{ in $\R^{d+1}$,}\\
U\ge0 &\text{ in $\R^{d+1}$,}\\
\{U=0\}\neq\emptyset\hem \text{ is compact,}
\end{cases}
$$
then we write 
\begin{equation}
\label{EqnSC}
U\in\mathcal{S}_c.
\end{equation} 
%Contact sets of solutions in $\SC$ are classified \cite{FSa}:
%\begin{thm}
%\label{ThmCompactContactOb}
%For $U\in\SC$, its contact set $\OLSet$ is either a point or an ellipsoid. 
%\end{thm} 

One fundamental ingredient  is the convexity of {global} solutions:
\begin{prop}[{\cite[Corollary 7]{C}}]
\label{PropObConvexity}
Suppose that $U$ is a global solution to the obstacle problem \eqref{OP}, then
$$
D^2u\ge0.
$$
\end{prop} 
Without tools from potential theory, we construct solutions in $\GC$ as {linearization}  limits of elements in $\SC$. This requires {sufficient} richness of $\SC$ as in the following theorem \cite{EW}:
\begin{thm}[{\cite[Lemma 3]{EW}}]
\label{ThmRichnessOfSC}
\mbox{}\\
\begin{enumerate}[1)]
\item \label{item:existence_of_ellipsoidal_solutions_by_ellipsoid} Given $\ell_j>0$ for $j=1,2,\dots,d+1$, and the corresponding ellipsoid 
$$
E= \Big \{\sum_{1\le j\le d} \frac{x_j^2}{\ell_j^2}+ \frac{y^2}{\ell_{d+1}^2}  \le1 \Big \},
$$
there is $U\in\SC$ such that 
$$
\{U=0\}=E.
$$
\vem

\item \label{item:existence_of_ellipsoidal_solutions_by_poly}  Given a {homogeneous quadratic polynomial} $P$ of the form
$$
P=\frac12\sum_{1\le j\le d}a_jx_j^2+\frac12 a_{d+1}y^2
$$ 
with $a_j>0$ and $\sum_{{j=1}}^{{d+1}} a_j=1$,
there is an ellipsoid $E$, symmetric with respect to each $\{x_j=0\}$ and $\HPP$, such that 
$$
U=P-1+V_E\in\SC
$$
and $$\{U=0\}=E.$$ 

Here $V_E$ denotes the Newtonian potential of $E$ as in \eqref{DefNT}.
\end{enumerate}
\end{thm} 
\begin{rem}
While Theorem \ref{ThmRichnessOfSC} \ref{item:existence_of_ellipsoidal_solutions_by_poly}) is exactly as stated in \cite[Lemma 3]{EW}, Theorem \ref{ThmRichnessOfSC} \ref{item:existence_of_ellipsoidal_solutions_by_ellipsoid}) is more classical. It is an old fact in potential theory that the Newton potential of an ellipsoid is a quadratic polynomial on the ellipsoid  \cite[§9, p 22]{Kellogg}. This potential of an ellipsoid can then be `completed' to a global solution to the obstacle problem as done in the proof of \cite[Lemma 3]{EW}.
\end{rem}

To get compactness of linearizations of the form $\frac{U-y^2/2}{\| U-y^2/2 \|_{L^2(\partial B_1)}}$ for $U\in\mathcal{S}_c$, we need the monotonicity formula by Weiss \cite{W}. For a function $U$, a point ${X}\in\R^{d+1}$ and a scale $r>0$, the \textit{ Weiss balanced energy} is defined as
\begin{equation}
\label{EqnWeiss}
\cW(U;{X},r):=\frac{1}{r^{d+3}}\int_{B_r({X})}(|\nabla U|^2+2U)-\frac{2}{r^{d+4}}\int_{\partial B_r({X})}U^2 {\dH{d}}.
\end{equation} 
When the center is the origin, we omit the dependence on the center from the expression.

For solutions in $\SC$, we have the following:
\begin{prop}[{\cite[Theorem 2]{W}}]
\label{PropWeiss}
Suppose $U\in\SC$, then 
$$
\frac{\dx{}}{\dx{r}} \cW(U;r)\ge0.
$$
The equality happens if and only if $U$ is $2$-homogeneous. 

Moreover, we have
$$
\lim_{r\to\infty}\cW(U;r)= \alpha_d \hem\text{ for  $U\in\SC$}
$$
where
$$\alpha_d:=\cW(\tfrac12y^2;0,1) \text{ is a dimensional constant}.$$ 
\end{prop} 
For the last convergence, we note that for a solution with a compact contact set, all its blow-downs are polynomial solutions, which have the same balanced energy $\alpha_d$. See, for instance, Lemma 3.31 in \cite{PSU}.

%%%%%%%%%%%%%%%%%%%%%%%%%%%%%%%%%%%%%%%%%%%%%%%%%%%%%%%%%%%%%%%%%%%%%%%%%%%%%%%%%%%%%%%%%%%%%%%%%%%%%%%%%%%%%%%%%%%%%%%%
\section{Solutions with prescribed contact sets} 
With our decomposition of the space $\R^{d+1}$ as in \eqref{EqnDecomposition}, the thin obstacle problem arises as the linearization of the classical obstacle problem around $y^2/2$ \cite{FiSe, SY3}. Let ${E'}$ denote a given thin ellipsoid in the hyperplane $\HPP$. To construct a solution $u$ to the thin obstacle problem with contact set $\LSet=E'$, it is natural to begin with a solution $U$ to the classical obstacle problem with $\{U=0\}=E$, where $E$ is a `thickening' of ${E'}$.

For a sequence of such solutions $U_n$, corresponding to ellipsoids $E_n$ with vanishing thickness,  the desired $u$ arises as the limit of normalized versions of $U_n-y^2/2$. To this end, we need the compactness of this family, which begins with the following:

\begin{lem}
\label{LemFrequencyBound}
For $U\in\SC$ and 
$$
W:=U- \tfrac{y^2}{2},
$$
we have
\begin{equation}
\label{EqnFrequencyBound}
R\int_{B_R}|\nabla W|^2\le 2\int_{\partial B_R}W^2 \dH{d} \qquad \text{ for all $R>0$}.
\end{equation}
If we further assume that
$$
\OLSet\subset B_{1/2},
$$
then we have
\begin{equation}
\label{EqnDoubling}
\frac{
\int_{\partial B_R}W^2 \dH{d}}{\int_{\partial B_1}W^2 \dH{d}} \le R^{\mu_d} \qquad \text{  for $R\ge 1$}
\end{equation}
where the constant $\mu_d$ is dimensional.
\end{lem} 
Recall the class of solutions $\SC$ from \eqref{EqnSC}, as well as the Weiss balanced energy from \eqref{EqnWeiss}.

\begin{proof}
	\mbox{}\\
\textit{Step 1: The frequency bound \eqref{EqnFrequencyBound}.}

With Proposition \ref{PropWeiss}, we have
\begin{align*}
0&=\lim_{r\to+\infty}\cW(U;r)-\cW(\tfrac{y^2}{2};1)
\ge \cW(U;R)-\cW(\tfrac{y^2}{2};R),
\end{align*}
where we used the constancy of $\cW$ in the $r$-variable at the $2$-homogeneous function $\tfrac{y^2}{2}$.
\\
Now we expand the last quantity as 
\begin{align*}
\cW(U;R)-\cW(\tfrac{y^2}{2};R)&=\frac{1}{R^{d+3}}\int_{B_R}|\nabla(U-\tfrac{y^2}{2})|^2+\nabla(y^2)\cdot\nabla(U-\tfrac{y^2}{2})\\
+&\frac{1}{R^{d+3}}\int_{B_R}2(U-\tfrac{y^2}{2})-\tfrac{2}{R^{d+4}}\int_{\partial B_R}(\UMY)^2+y^2(\UMY) \dH{d}.
\end{align*}
Integration-by-parts gives
$$
\int_{B_R}\nabla(y^2)\cdot\nabla(U-\tfrac{y^2}{2})=\int_{\partial B_R}\nabla(y^2)\cdot \frac{y}{R}(\UMY) \dH{d}-2\int_{B_R}(\UMY).
$$
Thus we have
$$
\cW(U;R)-\cW(\tfrac{y^2}{2};R)=\frac{1}{R^{d+3}}\int_{B_R}|\nabla(U-\tfrac{y^2}{2})|^2-\frac{2}{R^{d+4}}\int_{\partial B_R}(\UMY)^2 \dH{d}.
$$
Together with the definition of the function $W$, we have
$$
0\ge \cW(U;R)-\cW(\tfrac{y^2}{2};R)=\frac{1}{R^{d+3}}\int_{B_R}|\nabla W|^2-\frac{2}{R^{d+4}}\int_{\partial B_R}W^2 \dH{d}.
$$
This is the frequency bound \eqref{EqnFrequencyBound}. Similar observations have been made in \cite{M, FiSe}.
\\ \\
\textit{Step 2: The doubling estimate \eqref{EqnDoubling}.}

By definition of $W$ and \eqref{OP}, we have 
$$
W\Delta W=\frac{y^2}{2}\chi_{\OLSet},
$$
which is supported in $B_{1/2}$ by our assumption. Meanwhile, a direct computation gives
$$
\Delta(W^2)=2W\Delta W+2|\nabla W|^2\ge2W\Delta W.
$$

Take a cut-off function $\eta$ that is $1$ in $B_{1/2}$ and supported on a compact subset of  $B_R$, we have
\begin{align*}
\int_{B_R}W\Delta W&=\int_{B_{1/2}}W\Delta W\\
&\le\frac12\int_{B_{1/2}}\Delta(W^2)\\
&\le\frac12\int_{\R^{d+1}}\eta\Delta(W^2)\\
&=\frac12\int_{\R^{d+1}}W^2\Delta\eta.
\end{align*}
The cut-off function can be chosen such that $\Delta\eta\le \frac{C}{R^2}$ for a dimensional constant $C$. As a result, we have
$$
\int_{B_R}W\Delta W\le \frac{C}{R^2}\int_{B_R}W^2.
$$
With $W^2$ being subharmonic, the mean value property gives 
$
\int_{B_R}W^2\le CR\int_{\partial B_R}W^2 \dH{d},
$
leading to 
\begin{equation}
\label{EqnRatioBetweenCrossTermAndHeight}
\int_{B_R}W\Delta W\le \frac{CR}{R^2}\int_{\partial B_R}W^2 \dH{d}.
\end{equation}
for a dimensional constant $C$.
\\
Define $H(r):=\frac{1}{r^{d+4}}\int_{\partial B_r}W^2 \dH{d}$ and recall the Almgren frequency function as in \eqref{EqnAlmgren}, we use \eqref{EqnDOfHeight} with $\lambda_\infty=2$ to get
$$
\frac{\frac{\dx{}}{\dx{r}}H(r)}{H(r)}=\frac{2}{r}[\Phi(W;r)-2]+\frac{2\int_{B_r}W\Delta W}{\int_{\partial B_r}W^2 \dH{d}}\le \frac{2\int_{B_r}W\Delta W}{\int_{\partial B_r}W^2 \dH{d}},
$$
since $\Phi(W;r)\le 2$ by \eqref{EqnFrequencyBound}.
With \eqref{EqnRatioBetweenCrossTermAndHeight}, we get
$$
\frac{\frac{\dx{}}{\dx{r}}H(r)}{H(r)}\le \frac{Cr}{r^2} \text{ for $r\ge1.$}
$$
Integrating this relation gives the desired estimate \eqref{EqnDoubling}. 
\end{proof} 
This leads to the desired compactness:
\begin{cor}
\label{CorGradientBound}
For $U \in \SC$ and $W := U - \tfrac{y^2}{2}$with 
$$
\OLSet\subset B_{1/2},
$$
define 
$$
\hW:=   \frac{W}{\|W\|_{L^2(\partial B_1)}}.
$$
Then for each $R\ge 1$, there is a constant $C_R$, depending only on $R$ and the dimension $d$, such that 
$$
|\nabla\hW|+|\hW|\le C_R \text{ in $B_R$.}
$$
\end{cor}

\begin{proof}
Since both $U$ and $\frac{y^2}{2}$ solve the obstacle problem \eqref{OP}, we have
$$
\Delta W_+\ge0 \hem \text{ and }\hem\Delta W_-\ge 0.
$$
Combined with 
$$
\int_{\partial B_{2R}}W^2  {\dH{d}} \le (2R)^{\mu_d} \|W\|_{L^2(\partial B_1)}^2
$$
from Lemma \ref{LemFrequencyBound},
we have
$$
|W|\le C_R \|W\|_{L^2(\partial B_1)} \text{ in $B_{\frac32R}$}
$$
for a constant $C_R$, depending only on $R$ and $d$. 
This is the desired $L^\infty$-bound. 

For a unit vector $\tau$ in $\HPP$, Proposition \ref{PropObConvexity} implies the pure second derivative of $\hW$ satisfies
$$
\hW_{\tau\tau}\ge0.
$$
Together with the $L^\infty$-bound on $\hW$ in $B_{\frac32 R}$, this implies
\begin{equation}
\label{EqnEstimateTangentialDerivative}
\hW_\tau\le C_R \text{ in $B_R$}
\end{equation}
for a constant $C_R$ depending only on $R$ and $d$. 

On the other hand, with $\Delta(\UMY)\le0$, we have
$
\Delta \hW\le0.
$
Thus
$$
\hW_{yy}\le-\sum_{1\le j\le d}\hW_{x_jx_j}\le 0.
$$
This implies
\begin{equation}
\label{EqnEstimateNormalDerivative}
\hW_y\ge -C_R \text{ in $B_R$.}
\end{equation}
Combining \eqref{EqnEstimateTangentialDerivative} and \eqref{EqnEstimateNormalDerivative} with the symmetry of the problem, we have
$$
|\nabla\hW|\le C_R \text{ in $B_R$}
$$
for a constant $C_R$ depending only on $R$ and $d$. 
\end{proof} 

With these preparations, we state the main lemma in the construction of global solutions to the thin obstacle problem \eqref{TOP}:
\begin{lem}
\label{LemMainLemma}
Given a sequence $(\ell_{j,n})_{n\in\mathbb{N}}$ for each $j=1,2,\dots,d,$  and a sequence $(\ell_{y,n})_{n \in \N}$ such that 
$$
0<\ell_{j,n}<\frac12, \hem \ell_{j,n}\to \ell_j>0, \text{ and } \hem 0<\ell_{y,n}\to0,
$$
let $E_n$ denote the corresponding ellipsoids
$$
E_n:= \Big\{\sum_{1\le j\le d} \frac{x_j^2}{\ell_{j,n}^2}+\frac{y^2}{\ell_{y,n}^2}\le1 \Big\}.
$$
For $U_n\in\SC$ satisfying 
$$
\OLSetn=E_n,
$$
let 
$$
\hW_n:=\frac{\UnMY}{\|\UnMY\|_{L^2(\partial B_1)}}.
$$
Then, up to a subsequence, we have
$$
\hW_n\to u \hem\text{ locally uniformly in $H^1(\R^{d+1})\cap C^\alpha(\R^{d+1})$} $$
for all $\alpha\in(0,1)$,
where $u\in\GC$ satisfies
$$
\LSet=\Big \{\sum_{1\le j\le d}  \frac{x_j^2}{\ell_{j}^2}  \le1, y=0 \Big\}\subset\HPP.
$$
\end{lem} 
Recall the classes of solutions $\SC$ and  $\GC$ defined in \eqref{EqnSC} and Definition \ref{DefGC}. The contact set $\LSet$ is defined in \eqref{CSet}.

\begin{proof}
The sub-sequential convergence  of $(\hW_n)_{n \in \N}$ follows from Corollary \ref{CorGradientBound}. It remains to see that the limit $u$ is in $\GC$ and has the desired contact set. 
\\ \\
\textit{Step 1: The limit  $u$ is a  solution to  \eqref{TOP}.}

With $\Delta U_n\le 1$ and $\Delta \frac{y^2}{2}=1$, we have $\Delta\hW_n\le 0$. Locally uniform convergence implies
$$
\Delta u\le 0 \hem\text{ in $\R^{d+1}$.}
$$
With $U_n\ge0$ in $\R^{d+1}$, we have $\hW_n\ge0$ on $\HPP$. Thus 
$$
u\ge0\hem \text{ on $\HPP.$}
$$

Given a point $X \in\{u>0\}$, locally uniform convergence gives $U_n(X)>0$ for large $n$. Thus $\Delta U_n({X})=1$ for large $n$. This implies $\Delta\hW_n({X})=0,$ and consequently, 
$$
\Delta u({X})=0 \text{ at ${X}\in\{u>0\}$.}
$$ 
For any point $X\notin\HPP$, our assumption $\ell_{y,n}\to0$ means ${X}\notin\OLSetn$ for all large $n$. Therefore, 
$$
\Delta u({X})=0 \text{ for any ${X}\notin\HPP$.}
$$

As a result, the limit $u$ is a global solution to the thin obstacle problem \eqref{TOP}.
\\ \\
\textit{Step 2: The limit  $u$ has sub-quadratic growth.}

For $R>0$, convergence in $H^1(B_R)$ implies
$$
\int_{B_R}|\nabla u|^2\le\liminf \limits_{{n \to \infty}} \int_{B_R}|\nabla\hW_n|^2=\liminf\limits_{{n \to \infty}} \frac{\int_{B_R}|\nabla(\UnMY)|^2}{\|\UnMY\|_{L^2(\partial B_1)}^2}.
$$
With Lemma \ref{LemFrequencyBound}, we continue as
\begin{align}
\int_{B_R}|\nabla u|^2&\le\liminf \limits_{{n \to \infty}} \frac{2\int_{\partial B_R}|\UnMY|^2}{R\|\UnMY\|_{L^2(\partial B_1)}^2}=\frac{2}{R}\liminf \limits_{{n \to \infty}} \int_{\partial B_R}\hW_n^2 \dH{d} \\
&=\frac{2}{R}\int_{\partial B_R}u^2 \dH{d}.
\end{align}

Recall  Almgren's frequency function from \eqref{EqnAlmgren}, we have
$$
\Phi(u;R)\le 2 \hem\text{ for all $R>0$.}
$$
Applying Corollary \ref{CorGrowthRateFromAlmgren}, we get
$$
\frac{1}{R^{d+4}}\int_{\partial B_R}u^2 \dH{d}   \le \int_{\partial B_1}u^2  \dH{d}.
$$
As a consequence of the normalization we made for $\hW_n$ and the uniform convergence of $\hW_n\to u$ on $\partial B_1$, we have
\begin{equation}
\label{EqnUnitHeightForLimit}
\int_{\partial B_1}u^2  \dH{d}=1.
\end{equation}
Thus for all $R\ge 1$, we have
$$
\int_{\partial B_R}u^2  \dH{d}  \le R^{d+4}.
$$ 
Since $u\Delta u=0$, the function $u^2$ is subharmonic. The previous estimate implies
$$
\sup_{X\in\R^{d+1}}\frac{|u(X)|}{|X|^2+1}<\infty.
$$
\\ \\
\textit{Step 3: The limit $u$ has the desired contact set.}

Denote the limiting ellipsoid by ${E'}$, that is, 
$$
{E'}:=\Big\{\sum_{1\le j\le d} \tfrac{x_j^2}{\ell_{j}^2} \le1 {,y=0} \Big \}.
$$
With $\{\UnMY=0\}\supset E_n\cap\HPP\to {E'}$, we have
$$
\LSet\supset {E'}.
$$

On the other hand, similar arguments as in Step 1 give
\begin{equation}
\label{EqnHarmonicOutSideE}
\Delta u=0 \text{ outside ${E'}$. }
\end{equation}
With the convexity of $U_n$ as in Proposition \ref{PropObConvexity}, we have that $\hW_n$ is convex along all the directions in $\HPP$. This implies the same convexity for $u$, and the convexity of $\LSet.$

Suppose there is ${X}\in\LSet\backslash E'$, then
$$\LSet\supset\mathrm{Conv}(\{{X}\}\cup E'),$$
where $\mathrm{Conv}(\cdot)$ denotes the convex hull of a set. 
With \eqref{EqnHarmonicOutSideE}, we have
$$
u=\Delta u=0 \text{ in $\mathrm{Conv}(\{X\}\cup {E'})\backslash E'$.}
$$
 Lemma \ref{LemIdentifyDelta} implies
$$
u=\frac{\partial u}{\partial y}=0 \text{ in $\mathrm{Conv}(\{X\}\cup E')\backslash E'$.}
$$
Unique continuation for harmonic functions implies $u=0$ in $\R^{d+1}$, contradicting \eqref{EqnUnitHeightForLimit}.

Therefore we must have $\LSet={E'}.$
\end{proof} 

The part of Theorem \ref{IntroMainThm1} concerning the existence of solutions follows:
\begin{proof}[Proof of Theorem \ref{IntroMainThm1}]
Given $\ell_j>0$ {for $j= 1, \dots, d$} as in  Theorem \ref{IntroMainThm1}, we might assume 
$\ell_j<\frac12
$ for each $j=1,2,\dots,d$ by a rescaling. Now for each $n\in\mathbb{N}$, define the thickening of ${E'}$ as
$$
E_n:=\Big \{\sum_{1\le j\le d}  \frac{x_j^2}{\ell_j^2}+n^2y^2\le1 \Big \}.
$$

Let $U_n$ denote a solution to the classical obstacle problem, given by Theorem \ref{ThmRichnessOfSC}, with
$$
\OLSetn=E_n.
$$We get the desired solution to the thin obstacle problem by applying Lemma \ref{LemMainLemma}.
\end{proof} 
The part of Theorem \ref{IntroMainThm1} concerning the uniqueness of solutions is delayed to Section 6.

%%%%%%%%%%%%%%%%%%%%%%%%%%%%%%%%%%%%%%%%%%%%%%%%%%%%%%%%%%%%%%%%%%%%%%%%%%%%%%%%%%%%%%%%%%%%%%%%%%%%%%%%%%%%%%%%%%%%%%%%
\section{Solutions with prescribed polynomial expansions}
We move towards the classification of compact contact sets as in Theorem \ref{IntroMainThm2}. 

Recall the space of solutions to the thin obstacle problem $\GC$ from Definition \ref{DefGC}, the space of polynomials $\PC$ from \eqref{EqnPC}, and  the bijection between them in Theorem \ref{ThmBijection}. Given a polynomial $p\in\PC$, it suffices to construct a solution $u\in\GC$ with an ellipsoid as its contact set and $p$ as its leading order expansion.

In the classification of compact contact sets for the classical obstacle problem, this step was achieved by a non-trivial fixed point argument in DiBenedetto-Friedman \cite{DF}. Unfortunately, this argument relies crucially on the explicit relation between a set and its Newtonian potential, and has no clear analogue in the thin obstacle problem. 

Similar to the previous section, we construct our solution as the {linearization} limit of a sequence of solutions to the classical obstacle problem {around the polynomial $\tfrac{y^2}{2}$}. To avoid degeneracy of the contact sets, we need the following lemma:

\begin{lem}
\label{LemNonDeg}
Given small $\eps,\delta>0$, big $R>0$, and a polynomial $P$ of the form
$$
P=\tfrac12y^2+\tfrac12\delta \Big[\sum_{1\le j\le d}a_jx^2_j- y^2 \Big]
$$
with 
$$
a_j > 0 \text{ and }\sum_{1\leq j\leq d} a_j=1,
$$
suppose that $U$ solves the classical obstacle problem in $B_R$ with
$$
U\le P-\eps \text{ on $\partial B_R$,}
$$
then 
$$
\OLSet\supset  {(} B_{\frac{\eps}{\delta R}}\cap\HPP {)}.
$$
\end{lem} 

\begin{proof}
Given $\xi_j\in\R$ for $j=1,2,\dots,d$, define 
$$
Q=\tfrac12y^2+\tfrac12\delta \Big [\sum_{1\le j\le d}a_j(x_j-\xi_j)^2- y^2 \Big].
$$
This is a solution to the classical obstacle problem. Meanwhile, if $|\xi_j|\le\frac{\eps}{\delta R}$, then on $\partial B_R$ we have
\begin{align*}
Q-U&=Q-P+P-U\\
&\ge \tfrac12\delta\sum_{1 \leq j \leq d} a_j(\xi_j^2-2x_j\xi_j)+\eps\\
&\ge -\delta R\cdot \max|\xi_j|+\eps\\
&\ge0.
\end{align*}
By the comparison principle between solutions, we have
$$
U\le Q \text{ in $B_R$.}
$$
As a result,  
$$U(\xi_1,\xi_2,\dots,\xi_d,0)=0$$ as long as $|\xi_j|\le\frac{\eps}{\delta R}$ for each $j$. 
\end{proof} 
The potential term in the expansion from part 2) of Theorem \ref{ThmRichnessOfSC} has the following decay:
\begin{lem}
\label{LemDecayOfV}
Let the {homogeneous, quadratic polynomial} $P$,  the solution $U$, the ellipsoid $E$, and  the potential $V_E$ be as in the statement of Theorem \ref{ThmRichnessOfSC}. 

There is $R>0$, depending on the coefficients ${(a_j)_{j=1}^d}$ of $P$, such that 
$$
\OLSet\subset B_R,
$$and
$$
V_E({X})\le (m-1)^{1-d} \hem\text{ for ${X}\notin B_{mR}$ and $m>1$.}
$$ 
\end{lem} 

\begin{proof}
With the positivity of $V_E$, we see that $U>P-1.$ As a result, there is $R>0$, of the size $1/\sqrt{\min_{{j=1,\dots, d}} a_j}$,  such that 
$$
E=\OLSet\subset B_R.
$$
For $X=(x,y)$ outside $B_{mR}$ with $m> 1$, we have
\begin{align*}
V_E(X)&={\kappa_{d+1}} \int_{E}|X-Y|^{1-d}dY
\le {\kappa_{d+1}} (m-1)^{1-d}\int_ER^{1-d} dY\\
&\le{\kappa_{d+1}} (m-1)^{1-d}\int_E|Y|^{1-d}dY
=(m-1)^{1-d}V_E(0).
\end{align*}
To conclude, we note that $U(0)=0$ implies
$
V_E(0)=1.
$
\end{proof} 

With these preparations, we give the main lemma of this section:
\begin{lem}
\label{LemMainLemma2}
Given $a_j>0$ for $j=1,2,\dots,d$ with $\sum_{{j=1}}^{{d}} a_j=1$, and the corresponding polynomial
$$
p=\tfrac12\sum_{1\le j\le d}a_jx_j^2-\tfrac12 y^2,
$$
we can find $u\in\GC$ such that 
$$\LSet \text{ is a thin ellipsoid in $\HPP$,}$$ 
and that 
$$
u=p-1+v
$$
as in the expansion in Theorem \ref{ThmBijection}.
\end{lem} 

\begin{proof}
\mbox{}\\
\textit{Step 1: Preparation of solutions to the obstacle problem.}

For each $n\in\N$, {$n>1$}, define 
$$
P_n=\frac12y^2+\frac{1}{n}p.
$$
An application of (a rescaled version of) Theorem \ref{ThmRichnessOfSC} gives 
$
U_n\in\SC
$
of the form 
\begin{align} \label{eq:definition_of_U_n}
U_n=P_n {-\tfrac{1}{n}}+V_n,
\end{align}
where $\{U_n=0\}$ is an ellipsoid, to be denoted by $E_n$, that is symmetric with respect to the coordinate axes, and $V_n$ is the Newtonian potential of $E_n$. 

With $V_n>0$, we have $U_n> P_n {-\frac{1}{n}}$. Consequently, 
\begin{equation}
\label{EqnLocalizingInY}
E_n\subset \{P_n  {-\tfrac{1}{n}} \le 0\}\subset\{|y|\le\frac{2}{\sqrt{n}}\}.
\end{equation}
On the other hand, by (a rescaled version of) Lemma \ref{LemDecayOfV}, we find ${R_0}>0$, depending only on ${(a_j)_{j=1}^d}$, such that 
$$
E_n\subset B_{{R_0}}, 
$$
and
$$
U_n\le\frac12 y^2+\frac{1}{n}p-\frac{1}{n}+\frac{1}{n}(m-1)^{1-d} \text{ on $\partial B_{m{R_0}}$}
$$
if $m>1$.

We fix $m=4$ and get
$$
U_n\le \frac12 y^2+\frac{1}{n}p-\frac{1}{2n} \text{ on $\partial B_{4{R_0}}$}.
$$
An application of Lemma \ref{LemNonDeg}, with $\delta=\frac{1}{n}$ and $\eps=\frac{1}{2n}$, gives
$$
E_n\supset B_{\frac{1}{8{R_0}}}\cap\HPP.
$$
Combining this with \eqref{EqnLocalizingInY} and the symmetry of $E_n$, we have
$$
E_n=\Big \{\sum_{1\le j\le d} \tfrac{x_j^2}{\ell_{j,n}^2}+\tfrac{y^2}{\ell_{y,n}^2}\le 1 \Big \},
$$
where 
$$\ell_{y,n}\le \frac{C}{\sqrt{n}}\hem\text{ and } \frac{c}{{R_0}}\le\ell_{j,n}\le C$$ for dimensional constants $c$ and $C$. 

Up to a subsequence, we have 
$$
\ell_{y,n}\to 0, \text{ and }\ell_{j,n}\to\ell_j>0 \text{ for each $j=1,2,\dots, d$.}
$$
\\ \\
\textit{Step 2: Convergence of linearizations.}

Let 
$$
\hW_n:=\frac{\UnMY}{\|\UnMY\|_{L^2(\partial B_{2{R_0}})}}.
$$
 We apply (a rescaled version of) Lemma \ref{LemMainLemma} to get
\begin{align} \label{eq:convergence_of_W_hat_n}
\hW_n\to u \text{ locally uniformly in $C^\alpha(\R^{d+1})$}
\end{align}
for all $\alpha\in(0,1)$, where
 $u\in\GC$ satisfies
$$
\LSet=\Big \{\sum_{1\le j\le d} \tfrac{x_j^2}{\ell_{j}^2} \le1 \Big \}\subset\HPP,
$$
and
\begin{align} \label{eq:u_has_uniform_norm}
\| u\|_{L^2(\partial B_{R_0})} =1.
\end{align}

It remains to see that $u$ has the desired expansion.
\\ \\
\textit{Step 3: Identifying the expansion of $u$. }

By definition, we have
$$
\hW_n=\frac{1}{\|\UnMY\|_{L^2(\partial B_{2{R_0}})}}  \frac{1}{n} \Big [\frac12\sum_{1\le j\le d}a_jx_j^2-\frac12y^2-1 \Big ]  {+ \frac{1}{\|\UnMY\|_{L^2(\partial B_{2{R_0}})}} V_n}.
$$
{For notational simplicity let us set 
$$m_n := \|\UnMY\|_{L^2(\partial B_{2R})}.$$

First of all let us note that $\Big (\frac{1}{n} \frac{1}{m_n} \Big )_{n \in \N}$ is bounded.
If this was not true then, since $V_n >0$,
\begin{align}
\hW_n((x,0)) \geq \frac{1}{n} \frac{1}{m_n} (p((x,0)) -1) \to \infty \qquad \text{ for sufficiently large } |x|,
\end{align}
but this violates the locally uniform bound from Corollary \ref{CorGradientBound}.
\\ 
Hence, up to a subsequence, 
\begin{align} \label{eq:convergence_of_1_over_n_1_over_m_n}
	\frac{1}{n} \frac{1}{m_n} \to \beta \in [0,\infty) \quad \text{ as } n \to \infty.
\end{align}

Furthermore we have that $\beta >0$. 
Assume towards a contradiction that this was not true.
By the symmetry of $E_n$ if holds that $0 \leq V_{n}(X) \leq V_{n}(0)$ for all $X \in \R^{d+1}$. On the other hand \eqref{eq:definition_of_U_n} implies that
\begin{align}
0 = U_n(0) = - \frac{1}{n} + V_{n}(0)
\end{align}
and combining these it follows for every $X \in \R^{d+1}$ that
\begin{align} 
|\hW_n(X)| \leq \frac{1}{n} \frac{1}{m_n} |p(X) -1| + \frac{1}{m_n} V_{n}(0) \leq \frac{1}{n} \frac{1}{m_n} (|p(X)| +2) \to 0 \quad \text{ as } n \to \infty.
\end{align}
But this is a contradiction to \eqref{eq:u_has_uniform_norm}. 

The convergence \eqref{eq:convergence_of_W_hat_n} and  \eqref{eq:convergence_of_1_over_n_1_over_m_n} imply that 
\begin{align}
\frac{1}{m_n} V_n \to v \quad \text{ locally uniformly in } C^\alpha(\R^{d+1})
\end{align}
for any $\alpha \in (0,1)$.
}
%The locally uniform bound from Corollary \ref{CorGradientBound} implies that the coefficient $\frac{1}{n\|\UnMY\|_{L^2(\partial B_{2R})}}$ stays bounded. With the limit $u$ being nontrivial, the coefficient stays away from $0$. Up to a subsequence, we have
%$$
%\frac{1}{n\|\UnMY\|_{L^2(\partial B_{2R})}}\to\beta>0.
%$$
%
%With the convergence of $\hW_n$ and the fact that
%$$
%\frac{1}{n\|\UnMY\|_{L^2(\partial B_{2R})}}V_n=\hW_n-\frac{1}{n\|\UnMY\|_{L^2(\partial B_{2R})}}[\frac12\sum_{1\le j\le d}a_jx_j^2-\frac12y^2-1],
%$$
%we have 
%$$
%\frac{1}{n\|\UnMY\|_{L^2(\partial B_{2R})}}V_n\to \beta v 
%$$
% locally uniformly in $C^\alpha(\R^{d+1})$ for any $\alpha\in(0,1).$ 
Consequently, we have the expansion 
$$
u=\beta[p-1]+v.
$$ 
Dividing $u$ by $\beta$ and calling $\frac{1}{\beta} v$ again by $v$, we have
$$
u=p-1+{v}.
$$
The uniform decay estimate from Lemma \ref{LemDecayOfV} implies
$v\to 0$ at infinity. The fact $v=u-p+1$ implies that $v$ solves the problem in Theorem \ref{ThmBijection}.
\end{proof} 

Our classification result, Theorem \ref{IntroMainThm2}, follows:

\begin{proof}[Proof of Theorem \ref{IntroMainThm2}]
Given $u\in\GC$ with $\LSet$ having non-empty interior in the thin space $\HPP$, it has an expansion as in Theorem \ref{ThmBijection} of the form 
$$
u=p-1+v.
$$
See also {subsection} \ref{RemNormalizationOfBijection}.
\\
On the other hand, with Lemma \ref{LemMainLemma2}, we construct $\tilde{u}\in\GC$ with the expansion
$$
\tilde{u}=p-1+\tilde{v},
$$
and 
$$
\Lambda(\tilde{u}) \text{ is an ellipsoid.}
$$
The injectivity of {the function} $u\mapsto p-1$ in Theorem \ref{ThmBijection} implies $u=\tilde{u}$. As a result, the contact set $\LSet$ is an ellipsoid. 
\end{proof} 

%%%%%%%%%%%%%%%%%%%%%%%%%%%%%%%%%%%%%%%%%%%%%%%%%%%%%%%%%%%%%%%%%%%%%%%%%%%%%%%%%%%%%%%%%%%%%%%%%%%%%%%%%%%%%%%%%%%%%%%%%%%%%%%%%%%%%%%%%%%%%%%%%%%%%%%%%%%%%%%%%%%%%%%
\section{Uniqueness of solutions with common contact sets}
Suppose that $U$ is a solution to the classical obstacle problem \eqref{OP}, then there is a direct link between the contact set $\{U=0\}$ and the Laplacian of $U$, namely,
$$
\Delta U=\chi_{\{U=0\}}.
$$
From here, it is not difficult to show that there is at most one solution corresponding to a given contact set. 

For the thin obstacle problem \eqref{TOP}, this link is lost (see, for instance, Lemma \ref{LemIdentifyDelta}). As a result, very little is known about solutions with the same contact set. 
In our case, however, we can show that the space of solutions with the same contact set is one-dimensional. 

The following is the main result of this section and corresponds to the uniqueness part of Theorem \ref{IntroMainThm1}. Recall the class of solutions $\GC$ from Definition \ref{DefGC}, and the class of polynomials $\PC$ from \eqref{EqnPC}.

\begin{prop}
\label{PropUniqueness}
For $d\ge2$,
given $\ell_j>0$ for $j=1,2,\dots,d$ and the corresponding thin ellipsoid
$$
E':= \Big \{ (x,y) \in \R^{d} \times \R: \sum_{1\le j\le d}  \tfrac{x_j^2}{\ell_j^2}  \le1, y = 0 \Big \}\subset\HPP,
$$
suppose that  $u_1, u_2\in\GC$ satisfy  
$$\Lambda(u_1)=\Lambda(u_2)=E',$$ then there is $c\in\R$ such that 
$$
u_2=cu_1.
$$
\end{prop}
\begin{rem}
We thank one of the referees for pointing out this interesting question of uniqueness. 
\end{rem} 
For $u_1$ and $u_2$  as in Proposition \ref{PropUniqueness},  let $p_1$ and $p_2$ be their polynomial expansions at infinity as in Theorem \ref{ThmBijection}. It suffices to show that $p_2$ is a multiple of $p_1.$ 

The following lemma shows that the polynomial expansions inherit the symmetries of the contact sets:

\begin{lem}
\label{LemSymmetryOfPoly}
For $u\in\GC$ in $\R^{d+1}$ with $d\ge2$, let $p\in\PC$ be its polynomial expansion at infinity as in Theorem \ref{ThmBijection}.

If the contact set $\LSet$ is centered at the origin, then $\nabla p(0)=0$. 

If the contact set $\LSet$ is symmetric with respect to each of the coordinate axes, then $D^2p$ is diagonal. 
\end{lem} 
\begin{proof}
Up to a rescaling, we can assume that the polynomial $p\in\PC$ is of the form
$$
p(x,y)=\frac12 x\cdot Ax-\frac12y^2+b\cdot x-1,
$$
where $A$  is strictly positive definite.

Assuming that $\LSet$ is centered at the origin, we first show that $b=0$. 

To this end, we note that 
$$p(x,y)=\tilde{p}(x+A^{-1}b,y),$$ 
where
$$
\tilde{p}(x,y)=\frac12 x\cdot Ax-\frac12y^2-\frac12b\cdot A^{-1}b-1.
$$
With Lemma \ref{LemMainLemma2}, we find a solution $\tilde{u}\in\GC$ such that 
$$
|\tilde{u}-\tilde{p}|(X)\to0\hem \text{ as }|X|\to\infty. 
$$
In the proof of Lemma \ref{LemMainLemma2}, the contact sets of solutions to the obstacle problem were centered at the origin. As a result, we have that $\Lambda(\tilde{u})$ is centered at the origin. 
Define $w(x,y)=\tilde{u}(x-A^{-1}b,y)$, then $\Lambda(w)$ is centered at $A^{-1}b$.

Moreover, we have
$|w-p|(X)\to0$ as $|X|\to\infty$. As a result, we have $u=w$. Since $\LSet$ is centered at the origin, we have $b=0$.

With a similar argument using rotations instead of translations, it follows that $A$ is diagonal if $\LSet$ is symmetric with respect to the coordinate axes. 
\end{proof} 

With this preparation, we prove the main result of this section.
\begin{proof}[Proof of Proposition \ref{PropUniqueness}]
Up to a rescaling, we  assume $E'\subset B_1$. As a result, we have
$$
\Delta u_1=\Delta u_2=0 \hem\text{ in }B_1^c.
$$

Let $p_1$ and $p_2$ be the polynomial expansions at infinity for $u_1$ and $u_2$ as in Theorem \ref{ThmBijection}. With Lemma \ref{LemSymmetryOfPoly}, they are of the form 
$$
p_j(x,y)=\frac12\sum_{1\le k\le d}\alpha^j_k(x_k^j)^2-\beta^jy^2-c_j \hem\text{ for }j=1 \text{ and }2,
$$
where $\alpha_k^j>0$.

In particular, after the multiplication by a constant, we might assume 
$
\alpha_1^1=\alpha_1^2.
$ That is, 
\begin{equation}
\label{PointOfNormalization}
\frac{\partial}{\partial x_1}p_1=\frac{\partial}{\partial x_1}p_2 \hem\text{ in }\R^{d+1}.
\end{equation}

For $j=1,2$, we have $|u_j-p_j|(X)\to 0$ as $|X|\to\infty$, and that $u_j$ and $p_j$ are harmonic outside $B_1$. Standard elliptic estimate gives $|\frac{\partial}{\partial x_1}(u_j-p_j)|(X)\to0$ as $|X|\to\infty.$ With \eqref{PointOfNormalization}, we have 
$$
|\frac{\partial}{\partial x_1}(u_1-u_2)|(X)\to0 \hem\text{ as }|X|\to\infty.
$$
Together with $\frac{\partial}{\partial x_1}(u_1-u_2)=0$ on $E'$, we can apply the comparison principle to the harmonic function $\frac{\partial}{\partial x_1}(u_1-u_2)$ in $\R^{d+1}\backslash E'$ to obtain
$$
\frac{\partial}{\partial x_1}(u_1-u_2)=0 \hem\text{ in }\R^{d+1}.
$$
Consequently, the difference $u_1-u_2$ is independent of the $x_1$-variable. With $u_1=u_2$ on $E'$, we have 
$$
u_1-u_2=0 \hem\text{ on }\{x_1\in\R, y=0\}\cap\{|x_j|\le b_j \text{ for }j=2,3,\dots d\}.
$$
With Lemma \ref{LemIdentifyDelta}, we have
$$
\frac{\partial}{\partial y}(u_1-u_2)=0  \hem\text{ on }\{x_1\in\R, y=0\}\cap\{|x_j|\le b_j \text{ for }j=2,3,\dots d\}\backslash E'.
$$
Unique continuation principle for harmonic functions imply $u_1=u_2$ in $\R^{d+1}$.
\end{proof}

%%%%%%%%%%%%%%%%%%%%%%%%%%%%%%%%%%%%%%%%%%%%%%%%%%%%%%%%%%%%%%%%%%%%%%%%%%%%%%%%%%%%%%%%%%%%%%%%%%%%%%%%%%%%%%%%%%%%%%%%%%%%%%%%%%%%%%%%%%%%%%%%%%%%%%%%%%%%%%%%%%%%%%%%%%%%%%%%%%%%%%%%%%%%%%%%%%%%%%%%%%%%%%%%%%%%%%%%%%%%%%%%
\appendix
\section{Expansion of solutions  with compact contact sets and polynomial growth}
In \cite{ERW}, the authors constructed a bijection between a special class of polynomials and the space of solutions to the thin obstacle problem with compact contact sets and polynomial growth. Since they considered general fractional order elliptic operators, this bijection was constructed for solutions in $\R^{d+1}$ under the assumption that $d\ge3.$

In this appendix, we establish this bijection  for $d\ge 2$ when the operator is the Laplacian. Our proof differs from the one in \cite{ERW}.

\vem

Let $u$ be a global solution to the thin obstacle problem \eqref{TOP}. Recall the definition of its contact set $\LSet$ from \eqref{CSet}. If the solution $u$ satifies
\begin{equation}
\label{EqnPolyGrowth}
\LSet\neq\emptyset \text{ is compact, and }
\sup_{X\in\R^{d+1}}\frac{|u(X)|}{|X|^k+1}<\infty\hem \text{ for some }k\in\N,
\end{equation}
then we write
$$
u\in\GCk.
$$
Correspondingly, we consider the class of polynomials whose degrees are at most $k$ and  satisfy
$$
\Delta p=0 \text{ in }\R^{d+1},\hem  p(\cdot, y)=p(\cdot,-y) \text{ for all }y\in\R,
$$
$$
\text{ and } \hem\{x\in\R^d: \hem p(x,0)\le 0\}\neq\emptyset \text{ is bounded.}
$$ 
For a polynomial $p$ satisfying these conditions, we write
$$
p\in\PCk.
$$

The main result in this appendix is:
\begin{thm}
\label{ThmMainAppen}
Assuming $d\ge2,$ for each $u\in\GCk$, there is $p\in\PCk$ such that 
$$
u=p+v,
$$
where $v$ is the unique solution to the thin obstacle problem that vanishes at inifinity and has  $-p$ as the obstacle, namely, 
$$
\begin{cases}
\Delta v\le 0 &\text{ in $\R^{d+1}$,}\\
v\ge-p &\text{ on $\HPP$,}\\
\Delta v=0 &\text{ in $\{y\neq0\}\cup\{v>-p\},$}\\
v(\cdot,y)=v(\cdot,-y)&\text{ for all $y\in\R$,}\\
v(x,y)\to0 &\text{ as $|(x,y)|\to+\infty.$}
\end{cases}
$$

This map $p\mapsto u$ is a bijection between $\PCk$ and $\GCk$. 
\end{thm} 

The following lemma shows how to find the polynomial expansion for a given solution:
\begin{lem}
\label{LemSolToPolyAppen}
Suppose that $u\in\GCk$  in $\R^{d+1}$ with $d\ge2$, then there is a unique polynomial $p\in\PCk$ such that 
$$
|u-p|(X)\to 0 \text{ as }|X|\to\infty.
$$
\end{lem} 
\begin{proof}Once we find a polynomial $p$ as in the lemma, its uniqueness follows from the comparison principle for harmonic functions. Below we establish the existence of such a polynomial. 

Up to a rescaling, we  assume 
$
\LSet\subset B_1.
$ As a result, we have
$$
\Delta u=0 \text{ in }B_1^c.
$$
By the Liouville theorem for harmonic functions in exterior domains (see, for instance, Theorem 2.1 in \cite{LLY}), the growth rate  in \eqref{EqnPolyGrowth} gives a unique polynomial $p$ of degree at most $k$ such that 
\begin{equation}
\label{EqnVanishingAtInfitnityAppendix}
|u-p|=O(|X|^{1-d}).
\end{equation}

This polynomial $p$ satisifes
$$
\Delta p=0 \text{ in }\R^{d+1}.
$$
Uniqueness of this polynomial and our assumption $u(\cdot,y)=u(\cdot, -y)$ as in \eqref{TOP} imply that
$$
p(\cdot, y)=p(\cdot,-y) \text{ for all }y\in\R.
$$
With $|u-p|\to0$ at infinity,  the super harmonicity of $u$ and the harmonicity of $p$, the comparison principle gives
$
u\ge p \text{ in }\R^{d+1}.
$ As a result, our assumption that $\LSet\neq\emptyset$ implies that 
$$
\{x\in\R^d:\hem p(x,0)\le 0\}\neq\emptyset.
$$
It remains to show that this set is bounded in $\HPP$. 

To this end, we decompose $p$ into homogeneous terms
\begin{equation}
\label{EqnHomoExpansionOfP}
p=\sum_{0\le j\le m}q_j,
\end{equation}
where $m\le k$ is the degree of $p$, and each $q_j$ is homogeneous of degree $j$. 

For $R>0$, define 
$$
u_R(X):=\frac{u(RX)}{\|u(R\cdot)\|_{L^2(\partial B_1)}}.
$$
Along a subsequence of $R\to+\infty$, we have 
$$
u_R\to u_{\infty} \hem\text{ locally uniformly in }\R^{d+1},
$$
where $u_\infty$ is a homogeneous solution to \eqref{TOP} with 
$$
\Lambda(u_\infty)=\{0\}.
$$
Such solutions are classified by Garofalo-Petrosyan \cite{GP}, who showed that $u_\infty$ is a polynomial of even degree.

To be consistent with \eqref{EqnVanishingAtInfitnityAppendix} and \eqref{EqnHomoExpansionOfP}, we must have that $m$ is even, and that
$$
q_m=cu_\infty \text{ for some }c>0.
$$
Since $q_m$ is of the highest degree in the expansion of $p$ and $q_m>0$ in $\HPP\backslash\{0\}$, we have 
$$
\{x\in\R^d:\hem p(x,0)\le 0\} \text{ is compact}.
$$ 

Therefore, the polynomial $p$ belongs to the class $\PCk.$
\end{proof} 

The following lemma shows how to find a solution with a prescribed polynomial expansion:
\begin{lem}
\label{LemPolyToSolAppen}
Suppose that $p\in\PCk$ in $\R^{d+1}$ with $d\ge2$, then there is a unique solution $u\in\GCk$ such that 
$$
|u-p|(X)\to 0 \text{ as }|X|\to\infty.
$$
\end{lem} 
\begin{proof}
Once we find the solution $u$ as in the lemma, its uniqueness follows from the comparison principle between solutions to the thin obstacle problem. Below we show the existence of such a solution. 

Up to a rescaling, we assume 
$$
\{x\in\R^d:\hem p(x,0)\le0\}\subset B_1.
$$
The compactness of this set implies that 
$\inf_{x\in\R^{d}}p(x,0)$ is finite. As a result, we can find $m\in[0,+\infty)$ such that 
$$
p+m\ge0 \hem\text{ on }\HPP.
$$
In particular, $p+m$ is a solution to the thin obstacle problem \eqref{TOP}.

For each $R>1$, let $u_R$ be the solution to the thin obstacle problem in $B_R$ with $p$ as the boundary data, that is, 
$$\begin{cases}
\Delta u_R\le 0 &\text{ in }B_R,\\
u_R\ge 0 &\text{ on }B_R\cap\{y=0\},\\
\Delta u_R=0 &\text{ in }B_R\cap(\{y\neq0\}\cup\{u_R>0\}),\\
u_R=p &\text{ on }\partial B_R.
\end{cases}$$
The super harmonicity of $u_R$ and the harmonicity of $p$ imply, by the comparison principle, that
$$
u_R\ge p \hem\text{ in }B_R.
$$
In particular, we have 
$
\Lambda(u_R)\subset\{x\in\R^d:\hem p(x,0)\le0\}\subset B_1,
$
and as a consequence, 
$$
\Delta u_R=0 \hem\text{ in }B_1^c.
$$

On the other hand, since $p+m$ is a solution to the thin obstacle problem in $B_R$ with $p+m\ge u_R$ on $\partial B_R$, the comparison principle between solutions gives
$$
u_R\le p+m \hem\text{ in }B_R.
$$

For each $R>2$, define a barrier $\Psi_R$ as 
$$
\Psi_R(X)=\frac{m}{1-R^{1-d}}(\frac{1}{|X|^{d-1}}-\frac{1}{R^{d-1}}).
$$
Then $\Psi_R$ is harmonic in $B_R\backslash B_1$, and equals $m$ on $\partial B_1$ and $0$ on $\partial B_R$.  Applying the comparison principle between $u_R-p$ and $\Psi_R$ on $B_R\backslash B_1$, we have
$$
u_R-p\le \Psi_R\le 2m/|X|^{d-1} \hem\text{ in }B_R\backslash B_1.
$$

With $p\le u_R\le p+m$ in $B_R$, the family $\{u_R\}$ is locally compact. As a result, along a subsequence of $R\to+\infty$, we have
$$
u_R\to u \text{ locally uniformly in }\R^{d+1},
$$
where $u$ is a solution to the thin obstacle problem in $\R^{d+1}$. Moreover, this solution inherits the following bounds from $u_R$
$$
p\le u\le p+2m/|X|^{d-1}.
$$
In particular, we have 
$$
|u-p|(X)\to 0 \hem\text{ as }X\to \infty,
$$
and $u$ has polynomial growth of degree at most $k$.  
Moreover, the comparison $u\ge p$ implies that $\LSet\subset\{x\in\R^d:\hem p(x,0)\le0\}$ is compact. 

It remains to see that the contact set $\LSet$ is non-empty. Suppose it is empty, then $u$ is harmonic in $\R^{d+1}$. Then we must have $u=p$ in $\R^{d+1}$. This is a contradiction since we are assuming that $\{x\in\R^d:\hem p(x,0)\le0\}\neq\emptyset$ for $p\in\PCk.$
\end{proof} 

The main result Theorem \ref{ThmMainAppen} follows:
\begin{proof}[Proof of Theorem \ref{ThmMainAppen}]
Given $p\in\PCk$, Lemma \ref{LemPolyToSolAppen} maps it to a solution $u\in\GCk$. Lemma \ref{LemSolToPolyAppen} implies that this map is a bijection. 

Define $v$ as the difference between $u$ and $p$, namely, $v=u-p$. The equations satisfied by $v$ follow from \eqref{TOP} as well as the harmonicity of $p$. 
\end{proof}

%%%%%%%%%%%%%%%%%%%%%%%%%%%%%%%%%%%%%%%%%%%%%%%%%%%

\end{document}